\documentclass[runningheads]{MyArXiv}

\usepackage{graphicx} 

\usepackage{amsmath,amssymb}

\renewcommand{\thefootnote}

\title*{On asymptotic structure of the critical
Galton-Watson Branching Processes
with infinite variance and Immigration}

\titlerunning{\it Galton-Watson Branching Processes
with Immigration}

\author{Azam A.~Imomov \and Erkin E.~Tukhtaev}

\authorrunning{\it A.~Imomov \and E.~Tukhtaev}

\institute{
Department of Mathematics, Karshi State University, \\ 17, Kuchabag street, 180100 Karshi city, Uzbekistan\\
(E-mail: {\tt imomov{\_}\,azam@mail.ru, tukhtaev{\_}\,erkin@mail.ru})}

\begin{document}
\thispagestyle{empty}
\maketitle
\setlength{\leftskip}{0pt}
\setlength{\headsep}{16pt}
\begin{abstract}
    We observe the Galton-Watson Branching Processes.
    Limit properties of transition functions and their convergence to
    invariant measures are investigated.
\end{abstract}

\section{Introduction}

    Let $\left\{{X_n}, n \in \mathbb{N}_0 \right\}$ be the Galton-Watson Branching
    Process allowing Immigration (GWPI), where $\mathbb{N}_0 = \left\{0 \right\} \cup \mathbb{N}$
    and $\mathbb{N}= \left\{{1,2, \ldots} \right\}$. This is a homogeneous Markov chain with state
    space $\mathcal{S}\subset\mathbb{N}_0$ and whose transition probabilities are
\begin{equation*}
    p_{ij} = \textrm{coefficient of} \;\; s^{j} \;\; \textrm{in} \;\; h(s)\bigl( f(s)\bigr)^{i},
    \quad \parbox{2.4cm}{ {} $s \in [0, 1)$,}
\end{equation*}
    where $h(s)= \sum\nolimits_{j \in \mathcal{S}}{h_j s^j}$ and
    $f(s)= \sum\nolimits_{j \in \mathcal{S}}{p_j s^j}$ are probability generating functions (PGF's).
    The variable $X_n$ is interpreted as the population size in GWPI at the moment $n$.
    An evolution of the process will occurs by following scheme. An initial state is empty
    that is $X_0  = 0$ and the process starts owing to immigrants. Each individual at time
    $n$ produces $j$ progeny with probability $p_j$ independently of each other so that
    $p_0 >0$. Simultaneously in the population $i$ immigrants arrive with probability
    $h_i$ in each moment $n \in \mathbb{N}$. These individuals undergo further
    transformation obeying the reproduction law $\left\{{p_j}\right\}$ and $n$-step
    transition probabilities $p_{ij}^{(n)}:={\textsf{P}}\left\{{\left.{X_{n+k}=j}
    \right|X_k =i}\right\}$ for any $k \in \mathbb{N}$ are given by
\begin{equation}         \label{1}
    {\mathcal P}_{n}^{(i)}(s):= \sum\limits_{j \in {\mathcal S}}{p_{ij}^{(n)}s^j}
    = \bigl({f_n(s)}\bigr)^i \prod\limits_{k= 0}^{n-1}{h\bigl({f_k (s)}\bigr)}
    \quad \parbox{3cm}{\textit{for any} {} $i \in \mathcal{S}$,}
\end{equation}
    where $f_n (s)$ is $n$-fold iteration of PGF $f(s)$; see for example {\cite{Pakes79}}.
    Thus the transition probabilities $\left\{{p_{ij}^{(n)}} \right\}$ are completely
    defined by the probabilities $\left\{ {p_j } \right\}$ and $\left\{ {h_j } \right\}$.

    Classification of states of the chain $\left\{ {X_n } \right\}$ is one of fundamental
    problems in theory of GWPI. Direct differentiation of \eqref{1} gives
\begin{eqnarray*}
    {\textsf{E}} \bigl[{\left. {X_n} \right|X_0 = i} \bigr]
    = \left\{\begin{array}{l} \displaystyle {a} n + i  \hfill, \quad \parbox{2.8cm} {\textit{when} {} $m = 1 $,}    \\
    \\
    \displaystyle \left({{\displaystyle {a}
    \over \displaystyle {m - 1}} + i} \right)m^{n} -  {\displaystyle {a}
    \over \displaystyle {m - 1}} \;  \hfill, \quad \parbox{2.8cm} {\textit{when} {} $m \neq 1 $,}
    \end{array} \right.
\end{eqnarray*}
    where $m = f'(1-)$ is mean per-capita offspring number and $a = h'(1-)$. The received
    formula for ${\textsf{E}} \bigl[{\left. {X_n} \right|X_0 = i} \bigr]$ shows that classification
    of states of GWPI depends on a value of the parameter $m$. Process $\left\{{X_n} \right\}$
    is classified as sub-critical, critical and supercritical if $m < 1$, $m = 1$ and $m > 1$ accordingly.

    The above described population process was considered first by Heathcote {\cite{He65}} in 1965.
    Further long-term properties of $\mathcal{S}$ and a problem of existence and uniqueness of invariant
    measures of GWPI were investigated by Seneta {\cite{Sen69}},
    Pakes {\cite{Pakes71a}}, {\cite{Pakes71b}} and by many
    other authors. Therein some moment conditions for PGF $f(s)$ and $h(s)$ was required to be
    satisfied. In aforementioned works of Seneta the ergodic properties of $\left\{{X_n }\right\}$
    were investigated. He has proved that when $m \le 1$ the process $\left\{ {X_n } \right\}$
    has an invariant measure $\left\{{\mu _k,\;k \in {\mathcal S}}\right\}$ which is unique
    up to multiplicative constant. Pakes {\cite{Pakes71b}} have shown that in
    supercritical case $ {\mathcal S}$ is transient. In the critical case ${\mathcal S}$ can
    be transient, null-recurrent or ergodic.  In this case, if in addition to assume  that
    $2b: = f''(1-) < \infty $, properties of $ {\mathcal S}$ depend on value of parameter
    $\lambda = {a \mathord{\left/ {\vphantom {a {b}}}\right. \kern-\nulldelimiterspace} b}$:
    if $\lambda > 1$ or $\lambda  < 1$, then $ {\mathcal S}$ is transient or null-recurrent
    accordingly. In the case when $\lambda = 1$, Pakes {\cite{Pakes71a}}
    studied necessary and sufficient conditions for a null-recurrence property. Limiting distribution
    law for critical process $\left\{ {X_n } \right\}$ was found first by Seneta {\cite{Sen70}}.
    He has proved that the normalized process ${{X_n}\mathord{\left/{\vphantom {{X_n}{\left({b}{n}\right)}}}
    \right. \kern-\nulldelimiterspace} {\left({b}{n}\right)}}$ has limiting Gamma distribution
    with density function ${\Gamma ^{-1}(\lambda )} x^{\lambda - 1}e^{-x}$ provided that
    $0 < \lambda < \infty$, where $x > 0$ and $\Gamma ( * )$ is Euler's Gamma function.
    This result  has been established also by Pakes {\cite{Pakes71a}} without reference to Seneta.
    Afterwards Pakes {\cite{Pakes79}}, {\cite{PakesMB75}}, has obtained
    principally new results for all cases $m < \infty$ and $b =\infty$.

    Throughout the paper we keep on the critical case only and $b =\infty$. Our reasoning will bound
    up with elements of slow variation theory in sense of Karamata; see {\cite{SenetaRV}}. Remind
    that real-valued, positive and measurable function $L(x)$ is  said to be slowly varying (SV) at
    infinity if ${{L (\lambda x)} \mathord{\left/ {\vphantom {{L (\lambda x)}{L (x)}}}\right.
    \kern-\nulldelimiterspace}{L (x)}}\to 1$ as $x \to \infty$  for each $\lambda > 0$. We refer
    the  reader to {\cite{AsHer}}, {\cite{Bingham}} and {\cite{SenetaRV}} for more information.

    In second section we study invariant measures of the simple Galton-Watson (GW) Process.
    In third section the invariant properties of GWPI will be investigated.

\section{Invariant measures of GW Process}

    Let $\left\{{Z_n, n\in \mathbb{N}_0}\right\}$ be the simple GW Branching Process
    without immigration given by offspring PGF $f(s)$.
    Discussing this case we will assume that the
    offspring PGF $f(s)$ has the following representation:
\begin{equation*}
    f(s) = s + (1 - s)^{1 + \nu} \mathcal{L}\left( {{\,1\,
    \over {1 - s}}}\right),   \eqno[\textsf {$f_\nu$}]
\end{equation*}
    where $0 < \nu \leq 1$ and $\mathcal{L}(x)$ is SV at infinity. By the
    criticality of the process the condition $\left[{f_\nu} \right]$ implies that
    $b =\infty$. This includes the case $b<\infty$ when $\nu =1$ and
    $\mathcal{L}(t) \rightarrow {b}$ as $t \to \infty$.

    Consider PGF ${f_{n}(s)}:= \textsf{E}\left[s^{Z_n}
    \left| {Z_0 =1}\right.\right]$ and write ${R_{n}(s)}:= 1-f_n(s)$. Evidently ${Q_{n}}:=R_{n}(0)$
    is the survival probability of the process. By arguments of Slack {\cite{Slack72}} one can
    be shown that if the condition $\left[{f_\nu} \right]$ holds then
\begin{equation}              \label{2}
    {Q_n^{\nu}} \cdot \mathcal{L} \left( {{{1} \over {{Q_n}}}} \right) \sim {{1} \over {\nu n}}
    \quad \parbox{2.4cm} {\textit{as} {} $n \to \infty$.}
\end{equation}
    Slack {\cite{Slack72}} also has shown that
\begin{equation}              \label{3}
    {\mathcal U}_n (s): = {{f_n (s) - f_n (0)} \over {f_n (0) - f_{n - 1} (0)}} \longrightarrow U(s)
\end{equation}
    for $s \in [0, 1)$, where the limit function $U(s)$ satisfies the Abel equation
\begin{equation}              \label{4}
    U\left( f(s) \right) = U(s) + 1,
\end{equation}
    so that $U(s)$ is PGF of invariant measure for the GW process $\left\{{Z_n}\right\}$.
    Combining $\left[{f_\nu} \right]$, \eqref{2} and \eqref{3} and considering
    properties of the process $\left\{{Z_n}\right\}$ we have
\begin{equation*}
    {\mathcal U}_n (s) \sim {U}_n (s) :=\left[ {1 - {{R_n (s)} \over {Q_ n}}} \right]{\nu {n}}
    \quad \parbox{2.4cm}{\textit{as} {} $n \to \infty$.}
\end{equation*}
    So we proved the following lemma.

\begin{lemma}                 \label{MyLem:1}
    If the condition $\left[{f_\nu} \right]$ holds then
\begin{equation}              \label{5}
    R_n (s) = {{{\mathcal N}\left(n \right)} \over
    {\left(\nu n \right)^{{1 \mathord{\left/ {\vphantom {1 \nu }} \right.
    \kern-\nulldelimiterspace} \nu }} }} \cdot
    \left[ {1 - {{U_n (s)} \over {\nu n}}} \right],
\end{equation}
    where the function ${\mathcal N}(x)$ is SV at infinity and
\begin{equation}              \label{6}
    {\mathcal N}(n) \cdot \mathcal{L}^{{1 \mathord{\left/ {\vphantom {1 \nu }}
    \right. \kern-\nulldelimiterspace} \nu }} \left( {{{(\nu n)^{{1 \mathord{\left/
    {\vphantom {1 \nu }} \right. \kern-\nulldelimiterspace} \nu }}}
    \over {{\mathcal N}(n)}}} \right) \longrightarrow 1
    \quad \parbox{2.4cm} {\textit{as} {} $n \to \infty$,}
\end{equation}
    and the function $U_n (s)$ enjoys following properties:
\begin{itemize}
\item  $U_n (s) \longrightarrow {U}(s)$ as $n \to \infty$ so that the equation \eqref{4} holds;
\item  $ \lim _{s \uparrow 1} U_n (s) = \nu n$ for each  fixed $n \in \mathbb{N}$;
\item  $ U_n (0) = 0$ for each  fixed $n \in \mathbb{N}$.
\end{itemize}
\end{lemma}

    Evidently that this lemma is generalization of \eqref{2} and herein it
    established by more simple proof rather than as shown in {\cite{Imomov18}}.

    Further writing $\Lambda (y) = y^\nu {\mathcal L}\left({{1 \mathord
    {\left/{\vphantom {1 {y}}} \right. \kern-\nulldelimiterspace} y}} \right)$
    we consider the function
\begin{equation}              \label{7}
    {\mathcal M}_n (s):=1-{{\Lambda \bigl({R_n(s)}\bigr)}\over {\Lambda \left({Q_n}\right)}}.
\end{equation}
    It follows from \eqref{5} and from the properties of SV-function that
\begin{eqnarray*}
    {\mathcal M}_n (s)
    & = & 1-\left( {{{R_n(s)} \over {Q_n}}} \right)^\nu {{{\mathcal L}\bigl({{1 \mathord{\left/
    {\vphantom {1 {R_n (s)}}} \right.\kern-\nulldelimiterspace} {R_n (s)}}} \bigr)} \over
    {{\mathcal L}\bigl({{1 \mathord{\left/ {\vphantom {1 {Q_n}}}
    \right. \kern-\nulldelimiterspace}{Q_n}}}\bigr)}}            \\
\\
    & \sim &  1 - \left( {1 - {{U_n (s)} \over \nu n}} \right)^\nu
    \sim {{U_n (s)} \over n}\bigl( {1 + \rho _n (s)} \bigr)
    \quad \parbox{2.4cm} {\textit{as} {} $n \to \infty$,}
\end{eqnarray*}
    where $\rho _n (s) = {\mathcal O}\bigl( {{1 \mathord{\left/ {\vphantom {{1}{n}}} \right.
    \kern-\nulldelimiterspace} n}} \bigr)$ uniformly for all $s \in [0,1)$.

    Thus we obtain the following assertion.

\begin{lemma}            \label{MyLem:2}
    If the condition $\left[{f_\nu} \right]$ holds then
\begin{equation}              \label{8}
    n \cdot {\mathcal M}_n (s) \longrightarrow U (s)
    \quad \parbox{2.4cm} {\textit{as} {} $n \to \infty$,}
\end{equation}
    where $U(s)$ is PGF of invariant measure of GW Process.
\end{lemma}

    In the following Lemma we find out an explicit form of PGF of $U(s)$. Write
\begin{equation*}
    {\mathcal V}(s) = {1 \over {\nu \Lambda \left( {1 - s} \right)}}.
\end{equation*}

\begin{lemma}            \label{MyLem:3}
    If the condition $\left[{f_\nu} \right]$ holds then
\begin{equation}              \label{9}
    U(s) = {\mathcal V}(s) - {\mathcal V}(0).
\end{equation}
\end{lemma}
\begin{proof}
    In pursuance of reasoning from {\cite[p.~401]{Bingham}} we obtain the following relation:
\begin{equation*}
    {\mathcal V}\bigl({f_{n+1}(s)} \bigr)-{\mathcal V}\bigl({f_n(s)} \bigr) \longrightarrow 1
    \quad \parbox{2.4cm} {\textit{as} {} $n \to \infty$.}
\end{equation*}
    Thence summing by $n$ we find
\begin{equation*}
    {\mathcal V}\bigl( {f_n (s)} \bigr) - {\mathcal V}(s) = n\cdot \bigl( {1 + o(1)} \bigr)
    \quad \parbox{2.4cm} {\textit{as} {} $n \to \infty$.}
\end{equation*}
    Keeping our designation we easily will transform last equality to a form of
\begin{equation}              \label{10}
    \Lambda \bigl( {R_n (s)} \bigr) = {{\Lambda \left( {1 - s} \right)} \over
    {\Lambda \left( {1 - s} \right)\nu n + 1}}\bigl( {1 + o(1)} \bigr)
    \quad \parbox{2.4cm} {\textit{as} {} $n \to \infty$.}
\end{equation}
    Combining \eqref{7}, \eqref{8} and \eqref{10} we reach \eqref{9}.
\end{proof}

\section{Invariant measures of GWPI}

    Consider GWPI. Pakes {\cite{PakesMB75}} has proved the following theorem.

\medskip

\noindent{{\textbf{Theorem P1}} {\cite{PakesMB75}}.}
    {\it If $m = 1$ then}
\begin{equation*}
    p_{00}^{(n)} \sim K\exp \left\{{\int\limits_{1}^{e^n}
    {{{\ln h\left({1 - \varphi (y)} \right)} \over y}\,dy}} \right\}
    \quad \parbox{2.4cm}{\textit{as} {} $n  \to \infty$,}
\end{equation*}
    {\it where $\varphi (y)$ is decreasing SV-function. If}
\begin{equation*}
    \sum\limits_{m = 0}^\infty  {\Bigl[ {\left( {1 - h(f_m (0)}
    \right)\left( {1 - f'(f_m (0)} \right)} \Bigr]} < \infty,
\end{equation*}
    {\it then }
\begin{equation*}
    p_{00}^{(n)}  \sim K_1 \exp \left\{ {\int\limits_{0}^{f_n (0)}
    {{{\ln h(y)} \over {f(y) - y}}\,dy} } \right\}
    \quad \parbox{2.4cm}{\textit{as} {} $n  \to \infty$.}
\end{equation*}
    {\it Herein $K$ and $K_1 $ are some constants.}

\medskip

    Since this point we everywhere will consider the case that immigration
    PGF $h(s)$ has the following form:
\begin{equation*}
    1 - h(s) = (1 - s)^{\delta} \ell \left( {{\,1\,
    \over {1 - s}}}\right),   \eqno[\textsf {$h_\delta$}]
\end{equation*}
    where $0 < \delta \leq 1$ and $\ell(x)$ is SV at infinity.

    Our results appear provided that conditions $\left[{f_\nu} \right]$
    and $\left[{h_\delta} \right]$ hold and $\delta > \nu$. As it has been shown
    in {\cite{PakesMB75}} that in this case $\mathcal{S}$ is ergodic.
    Namely we improve statements of Theorem P1. Herewith we put forward
    an additional requirement concerning $\mathcal{L}(x)$ and ${\ell}(x)$.
    So since $\mathcal{L}(x)$ is SV we can write
\begin{equation*}
    {{\mathcal{L}\left( {\lambda x} \right)} \over {\mathcal{L}(x)}}
    = 1 + \alpha(x)     \eqno[\textsf {$\mathcal{L}_{\alpha}$}]
\end{equation*}
    for each $\lambda > 0$, where $\alpha (x) \to 0$ as $x \to \infty $. Henceforth we
    suppose that some positive function $g(x)$ is given so that $g(x) \to 0$ and
    $\alpha (x)={o}\bigl(g(x)\bigr)$ as $x \to \infty $. In this case $\mathcal{L}(x)$
    is called SV with remainder $\alpha (x)$; see {\cite[p.~185, condition SR3]{Bingham}}.
    Wherever we exploit the condition $\left[\mathcal{L}_{\alpha} \right]$ we will suppose that
\begin{equation}          \label{11}
    \alpha(x) = {o}\left( {{{\mathcal{L}\left(x\right)} \over {x^\nu }}} \right)
    \quad \parbox{2.4cm} {\textit{as} {} $x  \to \infty$.}
\end{equation}
    And also by perforce we suppose the condition
\begin{equation*}
    {{{\ell}\left( {\lambda x} \right)} \over {{\ell}(x)}}
    = 1 + \beta(x)     \eqno[\textsf {${\ell}_{\beta}$}]
\end{equation*}
    for each $\lambda > 0$, where
\begin{equation*}
    \beta(x) = {o}\left( {{{{\ell}\left(x\right)} \over {x^\delta}}} \right)
    \quad \parbox{2.4cm} {\textit{as} {} $x  \to \infty$.}
\end{equation*}

    Since $f_{n}(s)\uparrow 1$ for all $s \in [0, 1)$ in virtue of \eqref{1} it sufficiently
    to observe the case $i=0$  as $n \to \infty$. Write
\begin{equation*}
    {\mathcal P}_{n}(s)={\mathcal P}_{n}^{(0)}(s).
\end{equation*}

    The following theorem is generalization of the Theorem P1.

\begin{theorem}         \label{MyTh:1}
    Let conditions $\left[{f_\nu} \right]$, $\left[{h_\delta} \right]$
    hold. If $\delta  > \nu $ then
\begin{equation*}
    {\mathcal P}_n (s) \sim K(s)\exp \left\{ - {\int\limits_s^{f_n (s)}
    {{{1 - h(y)} \over{f(y)-y}}\Bigl[{1+ \delta (1-y)}\Bigr]dy}}\right\}
\end{equation*}
    as $n \to \infty $, where $K(s)$ is a bounded function for $s \in [0, 1)$ and
    $\delta (x)\to 0$ as $x\downarrow 0$. If in addition, the conditions
    $\left[\mathcal{L}_{\alpha} \right]$ and \eqref{10} are satisfied then
\begin{equation*}
    {\mathcal P}_n (s) \sim K(s)\exp \left\{ - {\int\limits_s^{f_n (s)}{{{1 - h(y)} \over
    {f(y)-y}}\Bigl[{1+{o}\bigl({{\Lambda}\left({1-y}\right)}\bigr)}\Bigr]dy}}\right\}
    \quad \parbox{2.4cm}{\textit{as} {} $n  \to \infty$.}
\end{equation*}
\end{theorem}

\begin{corollary}             \label{MyCor:1}
    Let conditions $\left[{f_\nu} \right]$, $\left[{h_\delta} \right]$ hold. If $\delta > \nu$ then
\begin{equation*}
    p_{00}^{(n)} \sim A \exp \left\{-{{{\mathcal{N}}^{\,\nu}(n)}\over{\delta - \nu}}
    \cdot {\ell}\left( {{{(\nu n)^{{1 \mathord{\left/{\vphantom {1 \nu }} \right.
    \kern-\nulldelimiterspace} \nu }}} \over {{\mathcal{N}}(n)}}} \right)\right\}
    \quad \parbox{2.4cm} {\textit{as} {} $n  \to \infty$,}
\end{equation*}
    where $A$ is a positive constant and ${\mathcal{N}}(x)$
    is SV at infinity defined in \eqref{6}.
\end{corollary}

    We make sure that at the conditions of second part of Theorem~\ref{MyTh:1} PGF ${\mathcal P}_n(s)$
    converges to a limit ${\pi}(s)$ which we denote by the power series representation
\begin{equation*}
    {\pi}(s)=\sum\limits_{j \in \mathcal{S}}{\pi_j}{s^j}.
\end{equation*}
    In our conditions we can establish a speed rate of this convergence.

\begin{theorem}             \label{MyTh:2}
    Let conditions $\left[{f_\nu} \right]$, $\left[{h_\delta} \right]$ hold and $\delta >\nu$.
    Then ${\mathcal P}_n (s)$ converges to ${\pi}(s)$ which generates the invariant measures
    $\left\{{\pi}_j\right\}$ for GWPI. The convergence is uniform over compact subsets of the
    open unit disc. If in addition, the conditions $\left[\mathcal{L}_{\alpha} \right]$,
    \eqref{10} and $\left[{\ell}_{\beta}\right]$ are fulfilled then
\begin{equation*}
    {\mathcal P}_n (s)={\pi}(s)\left(1+{\Delta_n (s)}{{\mathcal{N}}_{\delta}}
    \left({{1}\over {R_n (s)}}\right)\right),
\end{equation*}
    where ${\mathcal{N}}_{\delta}(x)={\mathcal{N}}^{\delta}(x) {\ell (x)}$,
    the function ${\mathcal{N}}(x)$ is defined in \eqref{6} and
\begin{equation*}
    \Delta_n (s) = {{1}\over {\delta -\nu}}\,{{{1}} \over {\bigl(\nu_{n}(s)\bigr)^{{{\delta}
    \mathord{\left/ {\vphantom {{\delta}{\nu}}} \right. \kern-\nulldelimiterspace} \nu}-1}}}
    - {{1+\nu} \over {2\nu}}\,{{\ln \bigl[\nu_n(s)\bigr]} \over {\bigl(\nu_{n}(s)\bigr)^{{{\delta}
    \mathord{\left/ {\vphantom {{\delta}{\nu}}} \right. \kern-\nulldelimiterspace} \nu}}}}
    \bigl(1+o(1)\bigr)
\end{equation*}
    as $n \to \infty$ and $\nu_n(s)=\nu{n}+{\Lambda}^{-1}\left({1-s}\right)$.
\end{theorem}

    The following result is direct consequence of Theorem~\ref{MyTh:2}.
\begin{corollary}
    If conditions of Theorem~\ref{MyTh:2} hold then
\begin{equation*}
    p_{00}^{(n)} ={\pi}_{0}\cdot\Bigl(1+{\Delta_n}{{\mathcal{N}}_{\delta}}\left(n\right)\Bigr),
\end{equation*}
    where ${\mathcal{N}_{\delta}(n)}$ is SV at infinity and
\begin{equation*}
    \Delta_n = {{1}\over {\delta -\nu}}\,{{{1}} \over {\left(\nu{n}\right)^{{{\delta}
    \mathord{\left/{\vphantom {{\delta}{\nu}}}\right. \kern-\nulldelimiterspace}\nu}-1}}}
    - {{1+\nu} \over {2\nu}}\,{{\ln n} \over {\left(\nu{n}\right)^{{{\delta}
    \mathord{\left/{\vphantom {{\delta}{\nu}}}\right. \kern-\nulldelimiterspace}\nu}}}} \bigl(1+o(1)\bigr)
    \quad \parbox{2.4cm} {\textit{as} {} $n  \to \infty$.}
\end{equation*}
\end{corollary}

\begin{remark}
    The analogous result as in Theorem \ref{MyTh:2}  has been proved
    in {\cite{Imomov15}} provided that $\delta =1$ and $f'''(1-)<\infty$.
\end{remark}


\begin{thebibliography}{99}

\bibitem{AsHer}
Asmussen~S., Hering~H.~(1983). {\sl Branching processes}. Birkh\"{a}user, Boston.

\bibitem{Bingham}
Bingham~N.~H., Goldie~C.~M., Teugels~J.~L.~(1987). {\sl Regular Variation}. Univ.
Press, Cambridge.

\bibitem{He65}
Heatcote~C.~R.~(1965). A branching process allowing
immigration. {\sl Jour. Royal Stat. Soc.} Vol.~{\bf B-27}, pp.~138--143.

\bibitem{Imomov18}
Imomov~A.~A.~(2018). On a limit structure of the Galton-Watson
branching processes with regularly varying generating functions.
{\sl Prob. and Math. Stat.}, {\bf available online, to appear}.

\bibitem{Imomov15}
Imomov~A.~A.~(2015). On long-time behaviors of states of
    Galton-Watson Branching Processes allowing Immigration.
{\sl J Siber. Fed. Univ.: Math. Phys.} Vol.~{\bf 8(4)}, pp.~394--405.

\bibitem{Pakes79}
Pakes~A.~G.~(1979). Limit theorems for the simple
branching process allowing immigration, I.
The case of finite offspring mean.
{\sl Adv. Appl. Prob.} Vol.~{\bf 11}, pp.~31--62.

\bibitem{PakesMB75}
Pakes~A.~G.~(1975). Some results for non-supercritical
Galton-Watson process with immigration.
{\sl Math. Biosci.} Vol.~{\bf 24}, pp.~71--92.

\bibitem{Pakes71a}
Pakes~A.~G.~(1971). On the critical Galton-Watson
process with immigration.
{\sl Jour. Austral. Math. Soc.} Vol.~{\bf 12}, pp.~476--482.

\bibitem{Pakes71b}
Pakes~A.~G.~(1971). Branching processes with immigration.
{\sl Jour. Appl. Prob.} Vol.~{\bf 8(1)}, pp.~32--42.

\bibitem{SenetaRV}
Seneta~E.~(1972). {\sl Regularly Varying Functions}. Springer, Berlin.

\bibitem{Sen70}
Seneta~E.~(1970).  An explicit-limit theorem for the
critical Galton-Watson process with immigration.
{\sl Jour. Royal Stat. Soc.} Vol.~{\bf B-32(1)}, pp.~149--152.

\bibitem{Sen69}
Seneta~E.~(1969). Functional equations and the
Galton-Watson process.
{\sl Adv. Appl. Prob.} Vol.~{\bf 1}, pp.~1--42.

\bibitem{Slack72}
Slack~R.~S.~(1972). Further notes on branching processes with mean 1.
{\sl Wahrscheinlichkeitstheor. und Verv. Geb.} Vol.~{\bf 25}, pp.~31--38.

\end{thebibliography}
\end{document}